\documentclass{amsart}
\usepackage{amssymb,amscd,amsthm,latexsym}
\usepackage[all]{xy}
\CompileMatrices

\newtheorem{theorem}{Theorem}[section]
\newtheorem{lemma}[theorem]{Lemma}
\newtheorem{proposition}[theorem]{Proposition}

\newtheorem*{property-P}{Property P}

\theoremstyle{definition}

\newtheorem{definition}[theorem]{Definition}

\newtheorem{remark}[theorem]{Remark}

\newtheorem{example}[theorem]{Example}

\def\comp{\raisebox{0.2mm}{\ensuremath{\scriptstyle\circ}}}

\renewcommand{\to}{\longrightarrow}

\newcommand{\Ker}{\ensuremath{\mathsf{Ker\,}}}

\newcommand{\Cc}{\ensuremath{\mathbb{C}}}

\newcommand{\Fc}{\ensuremath{\mathcal{F}}}

\newcommand{\Ac}{\ensuremath{\mathcal{A}}}

\newcommand{\Xc}{\ensuremath{\mathbb{X}}}

\newcommand{\Eq}{\ensuremath{\mathsf{Eq(\Ac)}}}

\newcommand{\Ext}{\ensuremath{\mathsf{Ext}}}
\newcommand{\Cov}{\ensuremath{\mathsf{Cov}}}

\newcommand{\Spl}{\ensuremath{\mathsf{Spl}}}
\newcommand{\Arr}{\ensuremath{\mathsf{Arr}}}

\newcommand{\CExt}{\ensuremath{\mathsf{CExt}}}

\newcommand{\Ec}{\ensuremath{\mathcal{E}}}
\newcommand{\E}{\ensuremath{\mathcal{E}}}

\newcommand{\Triv}{\ensuremath{\mathsf{Triv}}}

\newbox\skewpullbackbox
\setbox\skewpullbackbox=\hbox{\xy 0;<1mm,0mm>: \POS(4,0)\ar@{-} (-4,0) \ar@{-} (8,4)
\endxy}

\newbox\ksewpullbackbox
\setbox\ksewpullbackbox=\hbox{\xy 0;<1mm,0mm>: \POS(0,-8)\ar@{-} (0,-4) \ar@{-} (4,-4)
\endxy}

\newbox\pullbackbox
\setbox\pullbackbox=\hbox{\xy 0;<1mm,0mm>: \POS(4,0)\ar@{-} (0,0) \ar@{-} (4,4)
\endxy}

\newbox\pushoutbox
\setbox\pushoutbox=\hbox{\xy 0;<1mm,0mm>: \POS(0,4)\ar@{-} (0,0) \ar@{-} (4,4)
\endxy}

\begin{document}

\title{Higher central extensions in Mal'tsev categories\thanks{Dedicated to George Janelidze on the occasion of his sixtieth birthday.\\The author is post-doctoral fellow with Fonds voor Wetenschappelijk Onderzoek (FWO-Vlaanderen).}
}
\thanks{Dedicated to George Janelidze on the occasion of his sixtieth birthday}

\author{Tomas Everaert}
%

  \address{
Vakgroep Wiskunde \\
Vrije Universiteit Brussel \\
 Department of Mathematics  \\
 Pleinlaan 2\\
1050 Brussel \\
 Belgium.
 }
 \address{
    Institut de recherche en mathŽmatique et physique, UniversitŽ catholique de Louvain, chemin du cyclotron 2 bte L7.01.02,
1348 Louvain-la-Neuve, Belgium}

\email{teveraer@vub.ac.be}

 \date{\today}

\maketitle

\begin{abstract}
Higher dimensional central extensions of groups were introduced by G.~Janelidze as particular instances of the abstract notion of covering morphism from categorical Galois theory. More recently, the notion has been extended to and studied in arbitrary semi-abelian categories.  In this article, we further extend the scope to exact Mal'tsev categories and beyond.
\end{abstract}
\pagestyle{myheadings}{\markboth{TOMAS EVERAERT},\markright{}}

\section{Introduction}

A \emph{Galois structure} $\Gamma=(\Cc,\Xc,H,I,\eta,\E)$ \cite{Janelidze:pure,Janelidze-Kelly:reflectiveness} consists of a category $\Cc$, a full replete reflective subcategory $\Xc$ of $\Cc$ with inclusion functor $H\colon \Xc\to \Cc$ \footnote{The right adjoint $H\colon \Xc\to \Cc$ is usually not required to be fully faithful, in which case it is necessary to consider moreover a class $\Fc$ of morphisms in $\Xc$. Here we make the same restriction as in \cite{Janelidze-Kelly:reflectiveness}.}, left adjoint $I\colon \Cc\to\Xc$ and reflection unit $\eta_A\colon A\to HI(A)$ (for $A\in\Cc$), and a class $\Ec$ of morphisms in $\Cc$ such that
\begin{itemize}
\item[(E1)] $\E$ contains all isomorphisms;
\item[(E2)] $\E$ is pullback-stable, in the following sense: the pullback of a morphism in $\E$ along any morphism in $\Cc$ \emph{exists} and lies in $\E$;
\item[(E3)] $\E$ is closed under composition;
\item[(G)] $HI(\Ec)\subseteq \Ec$.
\end{itemize} 
Given a fixed Galois structure $\Gamma=(\Cc,\Xc,H,I,\eta,\E)$, and for any object $B\in \Cc$, we shall use the notation $(\Cc\downarrow B)$ for the full subcategory of the slice category $\Cc/B$ containing those $f\in \Cc/B$ which lie in the class $\Ec$. We shall often write $(A,f)$ instead of just $f$ in order to denote an object $f\colon A\to B$ of $(\Cc\downarrow B)$. A \emph{monadic extension} (or \emph{effective $\E$-descent morphism}) is an object $(E,p)\in (\Cc\downarrow B)$ such that the pullback functor $p^*\colon (\Cc\downarrow B)\to (\Cc\downarrow E)$ is monadic. A \emph{trivial covering} of $B$ is an object $(A,f)\in(\Cc\downarrow B)$ such that the induced commutative square 
\[
\xymatrix{
A \ar[r]^-{\eta_A} \ar[d]_f & HI(A)\ar[d]^{HI(f)}\\
B\ar[r]_-{\eta_B}& HI(B)}
\]
is a pullback. If $(E,p)\in(\Cc\downarrow B)$ is a monadic extension and $(A,f)\in(\Cc\downarrow B)$, then $(A,f)$ is said to be \emph{split} by $(E,p)$ if $p^*(A,f)$ is a trivial covering of $E$. The full subcategory of $(\Cc\downarrow B)$ containing those $(A,f)\in (\Cc\downarrow B)$ that are split by $(E,p)$ is denoted by $\Spl_{\Gamma}(E,p)$. A \emph{covering} of $B$ is an object $(A,f)\in (\Cc\downarrow B)$ which lies in $\Spl_{\Gamma}(E,p)$ for \emph{some} monadic extension $(E,p)$. The full subcategory of $(\Cc\downarrow B)$ of all coverings of $B$ is denoted by $\Cov_{\Gamma}(B)$, and the full subcategory of the arrow category $\Arr(\Cc)$ containing the coverings of \emph{every} $B\in\Cc$ by $\Cov_{\Gamma}(\Cc)$. Finally, a \emph{normal extension} is a monadic extension $(E,p)$ which lies in $\Spl_{\Gamma}(E,p)$. 
 
We are particularly interested in the following example of a Galois structure $\Gamma=(\Cc,\Xc,H,I,\eta,\E)$ \cite{Janelidze:pure}: $\Cc$ is the variety of groups, $\Xc$ is the variety of abelian groups, $H\colon \Xc\to\Cc$ is the inclusion functor hence $I\colon \Cc\to \Xc$ is the abelianisation functor and $\eta_A$ is the quotient map $A\to A/[A,A]$, and $\Ec$ is the class of surjective homomorphisms. In this case, every $p\in\Ec$ is a monadic extension.  A trivial covering is a surjective homomorphism $f\colon A\to B$ whose restriction to the commutator subgroups $[A,A]\to [B,B]$ is an isomorphism. The coverings are precisely the central extensions: surjective homomorphisms $f\colon A\to B$ whose kernel $\Ker(f)$ is contained in the center of $A$. Normal extensions and coverings coincide in this case. We shall denote the full subcategory of the arrow category $\Arr(\Cc)$ containing all surjective homomorphisms (=extensions) by $\Ext(\Cc)$. 

Of particular interest to us here is that the inclusion functor $H_1\colon \Cov_{\Gamma}(\Cc)\to \Ext(\Cc)$ has a left adjoint $I_1\colon \Ext(\Cc)\to  \Cov_{\Gamma}(\Cc)$; its value at a surjective homomorphism $f\colon A\to B$ is given by the induced homomorphism $A/[\Ker(f), A]\to B$  and the $f$-component of the reflection unit, $\eta^1_f\colon f\to H_1I_1(f)$, by the quotient map $A\to A/[\Ker(f), A]$. This fact allowed G.~Janelidze to introduce a notion of \emph{double central extension} of groups in \cite{Janelidze:double}. These are defined as the coverings with respect to the Galois structure $\Gamma_1=(\Cc_1,\Xc_1,H_1,I_1,\eta^1,\E^1)$, where $\Cc_1=\Ext(\Cc)$, $\Xc_1=\Cov_{\Gamma}(\Cc)$, $H_1$, $I_1$ and $\eta^1$ are as above, and $\Ec^1$ is a suitably defined class of morphisms in $\Ext(\Cc)$.

For an arbitrary Galois structure $\Gamma=(\Cc,\Xc,H,I,\eta,\E)$ it is not always true that every morphism $f\in \E$ admits a reflection in $\Cov_{\Gamma}(\Cc)$, however this is the case for  many important examples.  A comprehensive study of conditions on $\Gamma$ under which ``coverings are reflective'' can be found in  \cite{Janelidze-Kelly:reflectiveness}. For the example of groups considered here, we have an even stronger property: not only do we have a left adjoint for  $H_1\colon \Cov_{\Gamma}(\Cc)\to \Ext(\Cc)$, which gives rise to the Galois structure $\Gamma_1$ and a notion of double central extension, but now also the inclusion functor $H_2\colon \Cov_{\Gamma_1}(\Cc_1)\to \Ext(\Cc_1)$ admits a left adjoint. (Here we have written $\Ext(\Cc_1)$ for the full subcategory of $\Arr(\Cc_1)$ determined by the class $\E^1$.) This, in its turn, yields yet another Galois structure $\Gamma_2=(\Cc_2,\Xc_2,H_2,I_2,\eta^2,\E^2)$ with $\Cc_2=\Ext(\Cc_1)$, $\Xc_2=\Cov_{\Gamma_1}(\Cc_1)$ and $\Ec^2$ a suitably defined class of morphisms in $\Ext(\Cc_1)$, which now allows us to define \emph{triple central extensions}, and so on.  More specifically, as was proved by G.~Janelidze and presented in his talk  \cite{Janelidze:talk}, we have an infinite chain of  Galois structures $\Gamma_n=(\Cc_n,\Xc_n,H_n,I_n,\eta^n,\E^n)$ such that for each $n\geq 1$,  $\Cc_n$ is the full subcategory of $\Arr(\Cc_{n-1})$ determined by $\E^{n-1}$, and $\Xc_n$ consists exactly of the coverings with respect to $\Gamma_{n-1}$. (Here we assumed that $\Cc_0=\Cc$, $\E^0=\E$ and $\Gamma_0=\Gamma$.) Let us also mention here that, just as for the case $n=0$, we have for each $n\geq 1$ that every morphism in $\E^n$ is a monadic extension, and that every covering with respect to $\Gamma_n$ is a normal extension. 

The particular Galois structure $\Gamma$ considered above is, in fact,  the prototypical example of a large class of Galois structures studied in \cite{Janelidze-Kelly:central-extensions}: namely, those $\Gamma=(\Cc,\Xc,H,I,\eta,\E)$ which are \emph{admissible} in the sense explained below, for which $\Cc$ is an exact category in the sense of \cite{Barr}, $H\colon \Xc\to\Cc$ is the inclusion functor into $\Cc$ of a Birkhoff subcategory $\Xc$ (=a full reflective subcategory closed under subobjects and regular quotient objects) and $\E$ is the class of regular epimorphisms in $\Cc$. The exactness of $\Cc$ implies, in particular, that every regular epimorphism is a monadic extension. If $\Cc$ is, moreover, \emph{Mal'tsev}  \cite{Carboni-Kelly-Pedicchio,Carboni-Lambek-Pedicchio} (=the permutability condition $RS=SR$ is satisfied for every pair of internal equivalence relations $R$ and $S$ on an object of $\Cc$) then every covering is a normal extension (see \cite{Janelidze-Kelly:central-extensions}, where the assumption on $\Cc$ was actually slightly weaker). In this case we moreover have that every Birkhoff subcategory of $\Cc$ is admissible.

The coverings with respect to such a $\Gamma$ were called \emph{central extensions} in \cite{Janelidze-Kelly:central-extensions}. This choice of terminology is motivated not only by the example of groups considered above, but also by the fact that when $\Cc$ is any variety of $\Omega$-groups then the coverings with respect to $\Gamma$ are precisely the central extensions studied by A.~Fr\"ohlich and others, defined with respect to the subvariety $\Xc$. Also, when $\Cc$ is, more generally, a Mal'tsev variety and $\Xc$ is the subvariety of abelian algebras in $\Cc$, then the coverings with respect to $\Gamma=(\Cc,\Xc,H,I,\eta,\E)$ are precisely the central extensions arising from commutator theory in universal algebra: namely, those surjective homomorphisms $f\colon A\to B$ for which the commutator $[A\times_BA, A\times A]$ of the kernel congruence $A\times_BA$ of $f$ with largest congruence $A\times A$ on $A$ is trivial (see \cite{Janelidze-Kelly:Maltsev}). 

Now let us turn our attention back to the infinite chain of Galois structures $\Gamma_n$ ($n\geq 0$) described above. Its existence is not specific to that particular $\Gamma$: in \cite{EGV} it was shown that whenever $\Cc$ is a semi-abelian category and $\Xc$ is a Birkhoff subcategory of $\Cc$, then the corresponding Galois structure $\Gamma=(\Cc,\Xc,H,I,\eta,\E)$ (where $\Ec$ is the class of regular epimorphisms in $\Cc$) induces a similar chain of Galois structures $\Gamma_n=(\Cc_n,\Xc_n,H_n,I_n,\eta^n,\E^n)$ ($n\geq 1$)), i.e.~we have for each $n\geq 1$ that  $\Cc_n$ is the full subcategory of $\Arr(\Cc_{n-1})$ determined by $\E^{n-1}$, and $\Xc_n$ consists exactly of the coverings with respect to $\Gamma_{n-1}$. Note that since any such $\Gamma$ is of the type studied in \cite{Janelidze-Kelly:central-extensions}, it makes sense to call the coverings with respect to $\Gamma_{n-1}$  \emph{$n$-fold central extensions}, as in the case of groups.

The aim of the present article is to further extend the latter result by showing the existence of  such a chain of ``higher order'' Galois structures under conditions on the Galois structure  $\Gamma=(\Cc,\Xc,H,I,\eta,\E)$ which are in particular satisfied when  $\Cc$ is exact Mal'tsev, $\Xc$ is a Birkhoff subcategory of $\Cc$ with inclusion functor $H\colon \Xc\to \Cc$ and $\Ec$ is the class of regular epimorphisms in $\Cc$. This includes the situation considered in \cite{EGV}, but we retrieve also the notion of double central extension considered in \cite{Gran-Rossi:double} and \cite{Everaert-VdL:Cahiers} in the exact Mal'tsev context, and extend it to higher dimensions. 
To obtain our result, it will be sufficient to show the existence of the Galois structure $\Gamma_1=(\Cc_1,\Xc, I_1, H_1, \eta^1,\E^1)$ and to prove that it satisfies the same conditions as imposed on $\Gamma$. The existence of the $\Gamma_n$ ($n\geq 2)$ will then simply follow by induction. Under our assumptions we shall moreover have for each $n\geq 1$ that every morphism in $\E^n$ is a monadic extension, and that every covering with respect to $\Gamma_n$ (i.e. every ($n+1$)-fold central extension) is a normal extension. 

To conclude this introduction, let us remark that the importance of higher dimensional central extensions lies foremost in their use in non-abelian homological algebra, and in particular in their connection to the Brown-Ellis-Hopf formulae for the homology of a group \cite{Brown-Ellis}. For more details, see for instance \cite{EverHopf,Everaert-Gran-nGroupoids,EGV,GVdL2,Janelidze:Hopf,RVdL}.

\section{Reflectiveness of coverings}\label{sectionreflection}

\emph{Throughout this section, we will fix a Galois structure $\Gamma=(\Cc,\Xc,H,I,\eta,\E)$.}
\newline

We use the notation $\Ext(\Cc)$ for the full subcategory of the arrow category $\Arr(\Cc)$ whose objects are the morphisms in $\Ec$. Note, however, that we do not, initially, require every $f\in\Ec$ to be a monadic extension.

We are interested in conditions on $\Gamma$ under which a left adjoint exists for the inclusion functor $\Cov_{\Gamma}(\Cc)\to \Ext(\Cc)$. A property which plays an important role in this is \emph{admissibility}, hence we begin by recalling what it means. 

As before, we use the notation $(\Cc\downarrow B)$ for the full subcategory of the slice category $\Cc/B$ containing those $f\in\Cc/B$ which lie in $\E$.  Also, we write  $(\Xc\downarrow Y)$ for  the full subcategory of the slice category $\Xc/Y$ containing those $(X, \varphi)\in\Xc/Y$ for which $H(\varphi)\in\E$. Since $H(I(\E))\subseteq \E$, the reflector $I\colon \Cc\to \Xc$ extends, for any $B\in\Cc$, to a functor $I^B\colon (\Cc\downarrow B)\to (\Xc\downarrow I(B))$ in an obvious way. $I^B$ has a right adjoint $H^B\colon (\Xc\downarrow I(B))\to (\Cc\downarrow B)$, sending an object $(X,\varphi)\in (\Xc\downarrow I(B))$ to $(A,f)\in (\Cc\downarrow B)$ defined via the pullback
\[
\xymatrix{
A \ar[r] \ar[d]_{f} & H(X) \ar[d]^{H(\varphi)}\\
B \ar[r]_-{\eta_B} & HI(B).}
\]  

\begin{definition}
The Galois structure $\Gamma$ is called \emph{admissible} when the functors $H^B\colon {(\Xc\downarrow I(B))}  \rightarrow {(\Cc\downarrow B)}$ are fully faithful. 
\end{definition}

Let us also recall the following
\begin{lemma}\label{trivialstable}\cite[Proposition~2.4]{Janelidze-Kelly:reflectiveness}\
If $\Gamma$ is admissible then $I\colon \Cc\to \Xc$ preserves those pullbacks 
\begin{equation}\label{pullback}\vcenter{
\xymatrix{
D \ar[r] \ar[d]_h & A \ar[d]^f\\
C \ar[r]_g & B}}
\end{equation}
for which $f$ is a trivial covering.
\end{lemma}
Consequently, if $f$ in the pullback diagram (\ref{pullback}) is a trivial covering, then so is $h$. And, if we assume that every pullback of a monadic extension is a monadic extension (which is certainly the case if every morphism in the pullback-stable class $\Ec$ is a monadic extension, as we shall assume later on) then we moreover have the following: if, once again in the pullback diagram (\ref{pullback}), $f$ is split by a monadic extension $p\colon E\to B$, then  $h$ is split by $g^*(E,p)$.  Hence: 
\begin{lemma}\label{stablecoverings}
If $\Gamma$ is admissible and monadic extensions are pullback-stable, then coverings are pullback-stable as well. 
\end{lemma}

We return to our problem of finding a left adjoint for the inclusion functor $H_1\colon \Cov_{\Gamma}(\Cc)\to \Ext(\Cc)$. If it exists, then it will necessarily restrict, for every $B$, to a left adjoint for the inclusion functor $\Cov_{\Gamma}(B)\to (\Cc\downarrow B)$. This is a consequence of the fact that every identity morphism is a covering and $\Cov_{\Gamma}(\Cc)$ is a replete subcategory of $\Ext(\Cc)$ (see Corollary 5.2 in \cite{Im-Kelly}). Moreover, if we assume that $\Gamma$ is admissible and that monadic extensions are pullback-stable, so that coverings are pullback-stable, then the existence of the latter left adjoints is also sufficient for $H_1$ to admit a left adjoint (by Proposition 5.8 in \cite{Im-Kelly}). Thus, we shall be looking for conditions on $\Gamma$ under which, for every $B\in\Cc$, the inclusion functor $\Cov_{\Gamma}(B)\to (\Cc\downarrow B)$ admits a left adjoint.

As remarked in \cite{Janelidze-Kelly:reflectiveness}, it often happens that there exists, for an object $B$, a single monadic extension $p\colon E\to B$ which splits every covering of $B$, i.e.~such that $\Cov(B)=\Spl(E,p)$. For instance,  $p\colon E\to B$  may be such that $E$ is projective with respect to all monadic extensions. Clearly, in this case, it will suffice for us to find a left adjoint for the inclusion functor  $\Spl(E,p)\to (\Cc\downarrow B)$, for every monadic extension $p\colon E\to B$.
 
However, here it will be unnecessary to assume the existence of such monadic extensions $p\colon E\to B$ for which $\Cov(B)=\Spl(E,p)$, since the existence of left adjoints for the inclusion functors $\Spl(E,p)\to (\Cc\downarrow B)$ automatically implies that of left adjoints for the inclusion functors  $\Cov_{\Gamma}(B)\to (\Cc\downarrow B)$ in the particular situation we are interested in. Indeed, as already mentioned, we will be interested in cases where every $p\in\Ec$ is a monadic extension, and every covering is a normal extension. Under these assumptions, and if moreover $\Gamma$ is admissible, any reflection $(\bar{E},\bar{p})$ in  $\Spl(E,p)$ of an object $(E,p)\in (\Cc\downarrow B)$ is necessarily also the reflection in $\Cov(B)$ of $(E,p)$.  Indeed,  for every morphism $h\colon (E,p)\to (A,f)$ to a covering $f\colon A\to B$ of $B$ there is a unique morphism $(\bar{E},\bar{p})\to (A,f)$ rendering the diagram
\[
\xymatrix@=17pt{
(E,p) \ar[rr] \ar[dr]_h && (\bar{E}, \bar{p}) \ar@{.>}[ld]\\
& (A,f)&}
\]
commutative because the existence of the morphism $h\colon (E,p)\to (A,f)$, i.e.~of the commutative triangle
\[
\xymatrix{
E \ar[rr]^h \ar[rd]_p && A \ar[ld]^f\\
&B&}
\]  
forces $(A,f)$ to be split by $(E,p)$: since $f\colon A\to B$ is a covering, hence a normal extension, by assumption, we have that $f^*(A,f)$ is a trivial covering. Hence, the pullback-stability of trivial coverings (by Lemma \ref{trivialstable}) tells us that also $h^*f^*(A,f)\cong(fh)^*(A,f)=p^*(A,f)$ is a trivial covering, i.e.~$(A,f)\in \Spl(E,p)$.     

We have just proved the following lemma:
\begin{lemma}\label{problemreduced}
Assume that $\Gamma$ is admissible, that every morphism in $\Ec$ is a monadic extension and that every covering is a normal extension. Then $H_1\colon \Cov_{\Gamma}(\Cc)\to \Ext(\Cc)$ admits a left adjoint as soon as for every $p\colon E\to B\in\E$ the inclusion functor $\Spl(E,p)\to (\Cc\downarrow B)$ admits a left adjoint.
\end{lemma}

In \cite{Janelidze-Kelly:reflectiveness} various sufficient conditions were given for left adjoints for the inclusion functors $\Spl(E,p)\to (\Cc\downarrow B)$ to exist, covering many important examples. Here we shall be interested only in the following one: the preservation by the functor $I\colon \Cc\to \Xc$ of those pullbacks (\ref{pullback}) for which $f,g\in \Ec$ and $g$ is a split epimorphism. Note that this condition is slightly stronger than what was required in Proposition 5.2(b) in \cite{Janelidze-Kelly:reflectiveness}. In the ``absolute'' case, where $\Ec$ consists of all regular epimorphisms in $\Cc$, this stronger property was first considered 
 in relationship with central extensions in \cite{Gran:applications}.  In \cite{DD} the authors say in this case that the reflector $I\colon \Cc\to \Xc$ is \emph{regular} if moreover every reflection unit $\eta_A\colon A\to HI(A)$ is a regular epimorphism.

\begin{lemma}\label{lemmaJK}\cite[Proposition 5.2(b)]{Janelidze-Kelly:reflectiveness}\
If $\Gamma$ is admissible, and $I\colon \Cc\to \Xc$ preserves those pullbacks (\ref{pullback}) for which $f,g\in \Ec$ and $g$ is a split epimorphism, then the inclusion functor $\Spl(E,p)\to (\Cc\downarrow B)$ admits a left adjoint, for every monadic extension $p\colon E\to B$.
\end{lemma}

Hence, under the assumptions of Lemma's \ref{problemreduced} and \ref{lemmaJK}, the inclusion functor $H_1\colon \Cov_{\Gamma}(\Cc)\to \Ext(\Cc)$ has a left adjoint. Moreover, the condition that every covering is a normal extension follows from the other assumptions. Indeed, first of all we have:

\begin{lemma}
If  $I\colon \Cc\to \Xc$ preserves those pullbacks (\ref{pullback}) for which $f,g\in \Ec$ and $g$ is a split epimorphism, then every covering which is a split epimorphism is a trivial covering.
\end{lemma}
\begin{proof}
The proof of Proposition 4.5(2) in \cite{EGV} remains valid. However, one should replace $\Gamma_n$ by the Galois structure $\Gamma$ considered here and, in particular,  ``higher order extensions'' by morphism in $\E$. To avoid confusion, we recall the argument.

Let $f\colon A\to B$ be a covering which is also a split epimorphism. Then there exists a diagram
\[
\xymatrix{
P \ar[r] \ar[d]_g & A \ar[d]^f \ar[r]^-{\eta_A} & HI(A) \ar[d]^-{HI(f)} \\
E \ar[r]_p & B \ar[r]_-{\eta_B}& HI(B)}
\] 
in which the left hand square is a pullback, $p$ is a monadic extension and $g$ a trivial covering. We must prove that the right hand square too is a pullback. Let us write $Q=B\times_{HI(B)}HI(A)$ for the pullback of $\eta_B$ and $HI(f)$. To see that the canonical morphism $A\to Q$ is an isomorphism, we consider it as a morphism $(A,f)\to (Q,\eta_{B}^*(HI(f)))$ in $(\Cc\downarrow B)$ and note that its image by the functor $p^*\colon (\Cc\downarrow B)\to (\Cc\downarrow E)$ is an isomorphism. The latter is the case since in the diagram above the exterior rectangle is a pullback. Indeed, it coincides with the exterior rectangle in the diagram
\[
\xymatrix{
P \ar[d]_g \ar[r]^-{\eta_P} & HI(P) \ar[d]^{HI(g)} \ar[r] & HI(A) \ar[d]^{HI(f)}\\
E \ar[r]_-{\eta_E} & HI(E) \ar[r]_{HI(p)} & HI(B)}
\]
which is a pullback since both of its squares are: the left hand one because $g$ is a trivial covering, and the right hand square since it is the image by $HI$ of the left hand square in the previous diagram, which is a pullback preserved by $I$ (hence also by $HI$) by assumption. 

It suffices now to observe that $p^*\colon (\Cc\downarrow B)\to (\Cc\downarrow E)$ reflects isomorphisms since it is monadic.  
\end{proof}

If we moreover have that coverings are pullback-stable, as is certainly the case when $\Gamma$ is admissible and every morphism in $\E$ is a monadic extension (by Lemma \ref{stablecoverings}) then the kernel pair projections $A\times_BA\to A$ of any covering $A\to B$ are coverings as well and therefore trivial coverings by the previous lemma. Whence:
\begin{lemma}\label{centralisnormal}
Assume that $\Gamma=(\Cc,\Xc,I,H,\eta,\E)$ is admissible, that every morphism in the class $\E$ is a monadic extension, and that $I\colon \Cc\to \Xc$ preserves those pullbacks (\ref{pullback}) for which $f,g\in \Ec$ and $g$ is a split epimorphism. Then every covering is a normal extension.
\end{lemma}

Finally, Lemma's \ref{problemreduced}, \ref{lemmaJK} and \ref{centralisnormal} yield
\begin{theorem}\label{propositionreflectiveness}
Assume that $\Gamma=(\Cc,\Xc,I,H,\eta,\E)$ is admissible, that every morphism in the class $\E$ is a monadic extension, and that $I\colon \Cc\to \Xc$ preserves those pullbacks (\ref{pullback}) for which $f,g\in \Ec$ and $g$ is a split epimorphism. Then the inclusion functor $\Cov_{\Gamma}(\Cc)\to\Ext(\Cc)$ admits a left adjoint and, moreover, every covering is a normal extension.
\end{theorem}

\subsection{Proof of Lemma \ref{lemmaJK}}

Before going to the next section, it is useful to give a proof for Lemma \ref{lemmaJK}, as we shall not only be needing the existence, but also the construction of the left adjoints of the inclusion functors $\Cov_{\Gamma}(\Cc)\to \Ext(\Cc)$, later on. 
 
For this, note first of all that when $\Gamma$ is admissible, i.e. if for every $B$, the right adjoint $H^B$ in the adjunction
 \[
\xymatrix@=30pt{
{(\Cc\downarrow B) \, } \ar@<1ex>[r]_-{^{\perp}}^{I^B} & {\, (\Xc\downarrow I(B)) \, }
\ar@<1ex>[l]^{H^B} }
\]
is fully faithful,  then the (essential) image of any functor $H^B$ consists exactly of the trivial coverings of $B$, and the composite $H^BI^B$ provides a left adjoint for the inclusion functor $\Triv_{\Gamma}(B)\to (\Cc\downarrow B)$. (Here we have written $\Triv_{\Gamma}(B)$ for the full subcategory $(\Cc\downarrow B)$ consisting of all trivial coverings.) Hence we have that ``trivial coverings are reflective''. In order to deduce from this the reflectiveness of coverings or, more specifically, the existence of left adjoints for the inclusion functors  $\Spl(E,p)\to (\Cc\downarrow B)$, we need to say a few words about what it means for a morphism to be a monadic extension.

First of all we note that a left adjoint for the pullback-functor $p^*\colon (\Cc\downarrow B)\to (\Cc\downarrow E)$ exists for \emph{any} morphism $p\colon E\to B$ in $\E$, and is given by composition with $p$. We denote it $\Sigma_p$ and write $T^p=p^*\Sigma_p$  for the corresponding monad on $(\Cc\downarrow E)$. Next, we recall that a (downward) morphism 
\[
\xymatrix{
S \ar@{}[rd] \ar@<0.5 ex>[r] \ar@<-0.5 ex>[r] \ar[d]_h & C \ar[d]^g\\
R \ar@<0.5 ex>[r] \ar@<-0.5 ex>[r] & E }
\]
of internal equivalence relations is a \emph{discrete fibration} if  the commutative square involving the second projections (hence, also the square involving the first projections) is a pullback. For a fixed equivalence relation $R$, we write  $\Cc^{R}$ for the category of discrete fibrations $(g,h)\colon S\to R$, with the obvious morphisms, and $\Cc^{\downarrow R}$ (respectively, $\Cc^{\downarrow_t R}$) for the full subcategory consisting of  those $(g,h)\colon S\to R$ for which $g$ and $h$ lie in $\E$ (respectively, are trivial coverings).

Now, let us fix a morphism $p\colon E\to B$ in $\Ec$ and write $Eq(p)$ for the equivalence relation $\xymatrix{E\times_B E\ar@<0.5 ex>[r] \ar@<-0.5 ex>[r] & E}$ (i.e. the kernel pair of $p$).  By sending any morphism $f\colon A\to B$ in $\Ec$ to the discrete fibration displayed in the left hand side of the diagram
\begin{equation}\label{comparisondiagram}\vcenter{
\xymatrix{
E\times_B E\times_B A  \ar[d] \ar@<0.5 ex>[r] \ar@<-0.5 ex>[r] & E\times_BA  \ar[r] \ar[d] & A \ar[d]^f\\
E\times_B E \ar@<0.5 ex>[r] \ar@<-0.5 ex>[r] & E \ar[r]_p & B}}
\end{equation}
obtained by first pulling back $f$ along $p$, and next taking kernel pairs, we obtain a functor $K^p\colon (\Cc\downarrow B)\to \Cc^{\downarrow Eq(p)}$. It turns out  (see, for instance, \cite{Janelidze-Tholen:facets1,Janelidze-Tholen2}) that $\Cc^{\downarrow Eq(p)}$ is equivalent to the category $(\Cc\downarrow E)^{T^p}$ of (Eilenberg-Moore) algebras for the monad $T^p$, and that $K^p\colon (\Cc\downarrow B)\to \Cc^{\downarrow Eq(p)}$ corresponds, via this equivalence, to the comparison functor $K^{T^p}\colon (\Cc\downarrow B)\to (\Cc\downarrow E)^{T^p}$, whence: 

\begin{lemma}\label{monadicdiscreet}\cite{Janelidze-Tholen:facets1,Janelidze-Tholen2}
A morphism $p\colon E\to B$ in $\E$ is a monadic extension if, and only if, the functor $K^p\colon (\Cc\downarrow B)\to \Cc^{\downarrow Eq(p)}$ is an equivalence of categories. 
\end{lemma}

Suppose now that the conditions of Lemma \ref{lemmaJK} are satisfied, and let $p\colon E\to B$ be a monadic extension. Then, since trivial coverings are pullback-stable (by Lemma \ref{trivialstable}) and by definition of $\Spl(E,p)$, the equivalence $K^p\colon (\Cc\downarrow B)\to \Cc^{\downarrow Eq(p)}$ restricts to an equivalence $\Spl(E,p)\to \Cc^{\downarrow_t Eq(p)}$. Therefore, in order to find a left adjoint for the inclusion functor $\Spl(E,p)\to (\Cc\downarrow B)$ it suffices to find one for the inclusion functor $\Cc^{\downarrow_t Eq(p)}\to \Cc^{\downarrow Eq(p)}$. But we already know from above that ``trivial coverings are reflective'' so that the inclusion functors $\Triv_{\Gamma}(E)\to (\Cc\downarrow E)$ and $\Triv_{\Gamma}(E\times_BE)\to (\Cc\downarrow E\times_BE)$ have left adjoints $H^EI^E\colon (\Cc\downarrow E)\to \Triv_{\Gamma}(E)$ and $H^{E\times_BE}I^{E\times_BE}\colon (\Cc\downarrow E\times_BE)\to \Triv_{\Gamma}(E\times_BE)$, respectively. And, we have that these extend to a left adjoint for $\Cc^{\downarrow_t Eq(p)}\to \Cc^{\downarrow Eq(p)}$, as a consequence of the assumption on $I\colon \Cc\to \Xc$. To see this, consider a discrete fibration $(g,h)\colon S\to Eq(p)$ as depicted below. Applying to it the functor $HI\colon \Cc\to \Cc$ and then pulling back along the units $\eta_E\colon E\to HI(E)$ and $\eta_{E\times_BE}\colon E\times_BE\to HI(E\times_BE)$ yields the right hand diagram
\[
\xymatrix{
S \ar@{}[rd] \ar@<0.5 ex>[r]^{\pi_1} \ar@<-0.5 ex>[r]_{\pi_2} \ar[d]_h & C \ar[d]^g && \bar{S} \ar@<0.5 ex>[r] \ar@<-0.5 ex>[r] \ar[d]_{\bar{h}} & \bar{C} \ar[d]^{\bar{g}}\\
E\times_BE \ar@<0.5 ex>[r] \ar@<-0.5 ex>[r] & E && E\times_BE \ar@<0.5 ex>[r] \ar@<-0.5 ex>[r] & E}
\]
where $\bar{g}=H^EI^E(g)$ and $\bar{h}=H^{E\times_BE}I^{E\times_BE}(h)$. We claim that it is a discrete fibration of equivalence relations, hence an object of $\Cc^{\downarrow_t Eq(p)}$.
Since, in the left hand diagram, the commutative square involving the second projections is a pullback which is preserved
by $I$ (hence also by $HI$) by assumption, we clearly have that the corresponding commutative square in the right hand diagram is a pullback too. From this it follows easily that $\bar{S}$ is a relation on $\bar{C}$, which is both reflexive and symmetric. It is moreover transitive, since the pullback
\[
\xymatrix{
S\times_CS \ar[r] \ar[d] & S \ar[d]^{\pi_1}\\
S \ar[r]_{\pi_2} & C}
\]
is preserved by $I$ (hence by $HI$), by assumption. Thus we obtain a functor $\Cc^{\downarrow Eq(p)}\to \Cc^{\downarrow_t Eq(p)}$, which is clearly left adjoint to $\Cc^{\downarrow_t Eq(p)}\to \Cc^{\downarrow Eq(p)}$. This concludes the proof of Lemma \ref{lemmaJK}.

\subsection{Construction of the left adjoint for $H_1\colon \CExt_{\Gamma}(\Cc) \to\Ext(\Cc)$}

We have just described the left adjoints of the inclusion functors $\Spl(E,p)\to (\Cc\downarrow B)$. Tracing back our steps, we thus find a left adjoint for $H_1\colon \Cov_{\Gamma}(\Cc)\to \Ext(\Cc)$, sending a morphism $f\colon A\to B$ in $\E$ to the covering $\bar{f}\colon \bar{A}\to B$ obtained as follows (as depicted in the diagram below): first we pull $f\colon A\to B$ back along itself, and then we take kernel pairs. This yields the discrete fibration $K^f(A,f)$ to which we apply the functor $HI\colon \Cc\to\Cc$. Next we pull back along the units $\eta_A\colon A\to HI(A)$ and $\eta_{A\times_BA}\colon A\times_BA\to HI(A\times_BA)$ and obtain a discrete fibration in $\Cc^{\downarrow_t Eq(f)}$. The covering in $\Spl(A,f)$ which corresponds to this discrete fibration via the equivalence $\Spl(A,f)\cong \Cc^{\downarrow_t Eq(f)}$ is the desired $\bar{f}$.
\begin{equation}\label{bigdiagram}\vcenter{
\xymatrix@=1em{
A\times_BA\times_BA \ar@<.5 ex>[ddd]  \ar@<-.5 ex>[ddd] \ar[rrr] \ar[rdd] \ar[rrd] &&& HI(A\times_BA\times_BA) \ar[rdd]  \ar@<.5 ex>@{.>}[ddd]  \ar@<-.5 ex>@{.>}[ddd]  & \\
&& P_1  \ar@<.5 ex>@{.>}[ddd]  \ar@<-.5 ex>@{.>}[ddd]  \ar[ru] \ar[ld] &&\\
& A\times_BA \ar[rrr]  \ar@<.5 ex>[ddd]  \ar@<-.5 ex>[ddd]  &&& HI(A\times_BA)  \ar@<.5 ex>[ddd]  \ar@<-.5 ex>[ddd]  \\
A\times_BA \ar[ddd]_{\pi_2} \ar[rdd]_>>>>>>>>>>{\pi_1} \ar@{.>}[rrd]_{q} \ar@{.>}[rrr] &&& HI(A\times_BA) \ar[rdd]_{HI(\pi_1)} &\\
&& P_0 \ar@{.>}[ddd]^{p} \ar[ld]^{\bar{\pi}_1} \ar[ru] &&\\
&A \ar[rrr] \ar[ddd]^<<<<<<<<<<f &&& HI(A) \\
A \ar[rdd]_f \ar@{.>}[rrd] &&&&\\
&&\bar{A} \ar[ld]^{\bar{f}} &&\\
& B &&&
}}
\end{equation}

\section{Higher central extensions}\label{mainsection}
\emph{As before, we fix a Galois structure $\Gamma=(\Cc,\Xc,I,H,\eta,\E)$.}
\newline

We now wish to consider conditions on $\Gamma$, as well as define a class $\E^1$ of morphisms in $\Ext(\Cc)$, such that
\begin{enumerate}
\item
the conditions on $\Gamma$ imply the assumptions of Theorem \ref{propositionreflectiveness}. Hence, in particular, the inclusion functor $H_1\colon \Cov_{\Gamma}(\Cc)\to\Ext(\Cc)$ admits a left adjoint $I_1\colon \Ext(\Cc)\to \Cov_{\Gamma}(\Cc)$ with reflection unit $\eta^1$.
\item
$\Gamma_1=(\Cc_1,\Xc, I_1, H_1, \eta^1,\E^1)$ is a Galois structure satisfying the same conditions as $\Gamma$. 
\end{enumerate}
Hence, when $\Gamma$ satisfies these conditions, we will obtain, inductively, for each $n\geq 1$ a Galois structure $\Gamma_n=(\Cc_n,\Xc_n,H_n,I_n,\eta^n,\E^n)$ such that $\Cc_n=\Ext(\Cc_{n-1})$, $\E^n=(\E^{n-1})^1$ and $\Xc_n$ consists exactly of the coverings with respect to $\Gamma_{n-1}$.

We begin by defining the class $\E^1$ of morphisms in $\Ext(\Cc)$. If $a, b\in\Ec$ then a morphism $a\to b$ in $\Ext(\Cc)$ is a commutative square
\begin{equation}\label{doubleext}\vcenter{
\xymatrix{
A' \ar[d]_a \ar[r]^{f'} & B' \ar[d]^{b}\\
A \ar[r]_{f} & B}}
\end{equation} 
in $\Cc$.  Certainly, we want every morphism in $\Ec^1$ to admit pullbacks along arbitrary morphisms in $\Ext(\Cc)$. Since we want $\E^1$ (respectively,  $\Gamma_1$) to inherit the properties of $\E$ (respectively, $\Gamma$) it is natural to require moreover that such pullbacks be pointwise---the inclusion functor $\Ext(\Cc)\to \Arr(\Cc)$ preserves them---and that every morphism (\ref{doubleext}) in $\Ec^1$ has its ``horizontal'' arrows $f$ and $f'$ too in $\Ec$. We define $\Ec^1$ as the class of morphisms in $\Ext(\Cc)$ satisfying these properties.

\begin{proposition}\label{propositiondext}
The above defined class $\E^1$ consists exactly of those commutative squares
\begin{equation}\label{doubleext}\vcenter{
\xymatrix{
A' \ar[d]_a \ar[r]^{f'} & B' \ar[d]^{b}\\
A \ar[r]_{f} & B}}
\end{equation} 
of morphisms in $\E$ for which also the induced morphism $(a,f')\colon A'\to A\times_{B}B'$ to the pullback of $b$ and $f$ lies in $\Ec$. Hence, the class $\E^1$ is the same $\E^1$ as considered for instance in \cite{EGV} in a semi-abelian context.

\end{proposition}
\begin{proof}
Let us write $\E'$ for the class of morphisms in $\Ext(\Cc)$ just defined. We must prove that $\E'=\E^1$. 

The validity of the inclusion $\E'\subseteq \E^1$ is straightforward and has been observed many times before with varying assumptions on $\Cc$ and $\E$---see, for instance, Proposition 3.5(1) in \cite{EGV} for a proof. For $\E'\supseteq \E^1$ it suffices to note that for each commutative square (\ref{doubleext}) in $\E^1$, the induced morphism $A'\to P$ to the pullback $P=A\times_{B}B'$ of $b$ and $f$ is the pullback in $\Arr(\Cc)$ of the morphisms $(f',f)\colon a\to b$ and $(1_{B'},b)\colon 1_{B'}\to b$: 
\[
\xymatrix{
& P\ar[rr] \ar@{.>}[dd] && B' \ar[dd]^b\\
A' \ar@{=}[dd] \ar[ru] \ar[rr]^>>>>>>>{f'} && B' \ar@{=}[ru] \ar@{=}[dd] &\\
& A \ar@{.>}[rr]_<<<<<<f && B \\
A' \ar[ru]^a \ar[rr]_{f'} && B' \ar[ru]_b}
\] 
\end{proof}

Let us then consider the following list of conditions on $\Gamma$, which we do not assume to be admissible, a priori.

\begin{itemize}
\item[(E4)] If  $g\comp f\in \E$ then $g\in\E$.
\item[(E5)] Any split epimorphism in $\Ext(\Cc)$, i.e.~any diagram
\[
\xymatrix{A' \ar@<.5ex>[r]^-{f'} \ar[d]_-{a} & B' \ar[d]^-{b} \ar@<.5ex>[l]^{s'}\\
A \ar@<.5ex>[r]^-{f} & B \ar@<.5ex>[l]^s}
\]
in $\Cc$ for which $a,b\in\E$ and $f'\circ s'=1_{B'}$, $f\circ s=1_B$, $b\circ f'=f\circ a$ and $a\circ s'=s\circ b$, lies in $\E^1$.
\item[(M)] Every morphism in $\Ec$ is a monadic extension.
\item[(B)] $\Xc$ is a \emph{strongly $\E$-Birkhoff} subcategory of $\Cc$: for every $f\colon A\to B\in\E$, the induced commutative square 
\begin{equation}\label{birkhofsquare}\vcenter{
\xymatrix{
A \ar[r]^-{\eta_A} \ar[d]_f & HI(A)\ar[d]^{HI(f)}\\
B\ar[r]_-{\eta_B}& HI(B)}}
\end{equation}
is in $\Ec^1$.
\end{itemize}

Let $\Cc$ be a category with finite limits and coequalisers, and let $\Ec$ consist of all regular epimorphisms in $\Cc$. Then conditions (E1)--(E4) are satisfied if, and only if, $\Cc$ is regular. When $\Cc$ is regular, then condition (E5) holds if, and only if, $\Cc$ is Mal'tsev (see \cite{Bourn:Maltsev,Bourn2003}). Condition (M) is certainly satisfied if $\Cc$ is Barr-exact, but there are many examples of regular categories $\Cc$ in which (M) is valid but which fail to be Barr-exact (for instance, the quasi-variety of torsion-free abelian groups, or the category of topological groups, both of which are also Mal'tsev). 

As mentioned in the introduction, a full reflective subcategory $\Xc$ of $\Cc$ is a \emph{Birkhoff} subcategory if it is closed in $\Cc$ under subobjects and regular quotient objects which, according to \cite{Janelidze-Kelly:central-extensions}, is equivalent to the condition that for every regular epimorphism $f\colon A\to B$, the square (\ref{birkhofsquare}) is a pushout square of regular epimorphisms. Notice that any pullback square of regular epimorphisms in a regular category is necessarily a pushout, and that this implies that also every member of the class $\E^1$ must be a pushout square. When $\Cc$ is Barr-exact and Mal'tsev, one has the converse: every pushout square of regular epimorphisms is in $\E^1$ (by Theorem 5.7 of \cite{Carboni-Kelly-Pedicchio}) hence every Birkhoff subcategory is strongly $\E$-Birkhoff. In particular when $\Cc$ is a Mal'tsev variety then condition (B) is satisfied if, and only if, $\Xc$ is a subvariety.

Hence the list of conditions (E3)--(E5), (M) and (B) is satisfied in particular in the situation studied in \cite{EGV}, i.e. for $\Cc$ a semi-abelian (hence exact Mal'tev) category, $\Xc$ a Birkhoff subcategory of $\Cc$ and $\Ec$ the class of all regular epimorphisms in $\Cc$. This includes the prototypical example where  $\Cc$ is the variety of groups and $\Xc$ is the variety of abelian groups. Some other examples are the following:: 

\begin{example}
As just explained, the conditions (E3)--(E5), (M) and (B) are satisfied for $\Cc$ any exact Mal'tsev category, $\Xc$ any Birkhoff subcategory of $\Cc$  and $\Ec$ the class of all regular epimorphisms in $\Cc$. For instance, we may choose $\Xc$ to be the Birkhoff subcategory of \emph{abelian objects}: objects $A\in\Cc$ which admit a (necessarily unique) internal Mal'tsev operation (=a morphism $p\colon A\times A\times A\to A$ satisfying $p(x,x,y)=y$ and $p(x,y,y)=x$).

An example of this situation outside the scope of \cite{EGV} is the following: for a fixed set $B$, let $\Cc$ be the category of (small) groupoids with set of objects $B$, whose morphisms are the functors $F\colon G\to H$ such that $F(b)=b$ for every object $b\in B$. Then $\Cc$ is indeed exact Mal'tsev (it is even \emph{protomodular}---see \cite{Bourn1991}), but not pointed, hence not semi-abelian, unless $B$ is a one-element set, in which case $\Cc$ is isomorphic to the variety of groups. The subcategory $\Xc$ of abelian objects of $\Cc$ consists exactly of those groupoids whose groups of automorphisms are abelian (see Theorem 4 in \cite{Bourn:Aspherical}). 
\end{example}

\begin{example}
Another example of a non-pointed exact Mal'tsev category $\Cc$ with a Birkhoff subcategory $\Xc$ is the variety $\Cc$ of commutative rings (with unit) with its subvariety $\Xc$ of Boolean rings (=rings $A$ such that $a^2=a$ for every $a\in A$). 
\end{example}


We want to prove that the above conditions imply the assumptions of Theorem \ref{propositionreflectiveness} and are inherited by the induced Galois structure $\Gamma_1$. We already know the following: 
\begin{lemma}\label{extensionsgoup}\cite[Proposition~3.4]{EJV}\
If the class $\E$ of morphisms in the category $\Cc$ satisfies conditions (E1)--(E5), then the class $\E^1$ of morphisms of $\Ext(\Cc)$ satisfies these same conditions. 
\end{lemma}

%
%

Before continuing, we recall that, for any monadic extension $p\colon E\to B$, the corresponding morphism $p\colon (E,p)\to (B,1_B)$ in $(\Cc\downarrow B)$ is a regular epimorphism. In fact, it follows immediately from this that $p$ is not only a regular epimorphism \emph{in $(\Cc\downarrow B)$}, but in \emph{any} $(\Cc\downarrow C)$ as soon as a morphism $g\colon B\to C$ exists in $\E$. More precisely, we have for any $g\colon B\to C$ in $\E$ that $p\colon (E,gp)\to (B,g)$ is a regular epimorphism in $(\Cc\downarrow C)$, simply because it is the image of $p\colon (E,p)\to (B,1_B)$ by the left adjoint functor $\Sigma_g\colon(\Cc\downarrow B)\to (\Cc\downarrow C)$.

\begin{lemma}\label{descentgoesup}
If the class $\E$ of morphisms in the category $\Cc$ satisfies conditions (E1)--(E5) as well as (M), then the class $\E^1$ of morphisms of $\Ext(\Cc)$ too satisfies condition (M). 
\end{lemma}
\begin{proof}
By Lemma \ref{monadicdiscreet}, the result will follow if we can prove, for any $(p',p)\colon e\to b$ in $\Ec^1$, that the functor 
\[
K^{(p',p)}\colon (\Ext(\Cc)\downarrow b)\to \Ext(\Cc)^{\downarrow\Eq(p',p)}
\]
is a category equivalence. We shall do this by constructing an essential inverse for it.

Let us write $\pi_i\colon R\to E$ and $\pi'_i\colon R'\to E'$  ($i\in\{1,2\}$) for the kernel pair projections of $p$ and $p'$, respectively. Consider an arbitrary object  of $\Ext(\Cc)^{\downarrow\Eq(p',p)}$, depicted as the left part of the diagram
\[
\xymatrix{
& S \ar@<0.5 ex>[rr]^{\rho_1} \ar@<-0.5 ex>[rr]_{\rho_2} \ar@{.>}[dd] && C \ar[dd]_>>>>>>g \ar@{.>}[rr]^q && A \ar@{.>}[dd]^f \\
S'  \ar@<0.5 ex>[rr]^>>>>>>>>>{\rho_1'} \ar@<-0.5 ex>[rr]_>>>>>>>>>{\rho_2'}  \ar[dd] \ar[ru] && C' \ar[dd]_>>>>>>>{g'} \ar[ru] \ar@{.>}[rr]^>>>>>>{q'} &&A'\ar@{.>}[dd]^<<<<<<<{f'} \ar@{.>}[ru]_a &\\
& R \ar@<0.5 ex>@{.>}[rr]  \ar@<-0.5 ex>@{.>}[rr]  && E\ar[rr]_<<<<<<<<p && B\\
R' \ar@<0.5 ex>[rr]^{\pi_1'} \ar@<-0.5 ex>[rr]_{\pi_2'}  \ar@{.>}[ru] && E' \ar[ru]_e \ar[rr]_{p'} && B'\ar[ru]_b}
\]
Since $p$ and $p'$ are monadic extensions, the functors $K^p\colon (\Cc\downarrow B)\to \Cc^{\downarrow \Eq(p)}$ and $K^{p'}\colon (\Cc\downarrow B')\to \Cc^{\downarrow \Eq(p')}$ have essential inverses from which we obtain the morphisms $f$ and $f'$ in the diagram, both in $\E$, and then also morphisms $q$ and $q'$ in $\E$ such that the squares $fq=pg$ and $f' q'=p' g'$ are pullbacks. Since $q'$ is a monadic extension, the corresponding morphism $q'\colon (C',q')\to (A',1_{A'})$ in $(\Cc\downarrow A')$ is the coequaliser of its kernel pair $(\rho_1,\rho_2)\colon (S',q'\rho_1')\to (C',q')$. Hence, its image $q'\colon (C',bf'q')\to (A',bf')$  by the functor  $\Sigma_{bf'}\colon (\Cc\downarrow A')\to (\Cc\downarrow B)$ is the coequaliser of $(\rho_1,\rho_2)\colon (S',bf'q'\rho_1')\to (C',bf'q')$, from which we obtain the morphism $a$ completing the right hand part of the diagram. Since $aq'=qc$ lies in $\Ec$ as a composite of morphisms in $\E$, we also have $a\in\Ec$, by condition (E4). Moreover, since the class $\E^1$ of morphisms in $\Ext(\Cc)$ too satisfies conditions (E3) and (E4) (by Lemma \ref{extensionsgoup}), the square $bf'=fa$ lies in $\E^1$ because the squares $bp'=pe$ and $eg'=gc$ do so by assumption. Hence, the square $bf'=fa$ is an object of $(\Ext(\Cc)\downarrow b)$, and we have defined the essential inverse for $K^{(p',p)}$ on objects. 

Moreover, any morphism in $\Ext(\Cc)^{\downarrow\Eq(p',p)}$ induces, via the essential inverses for $K^p\colon (\Cc\downarrow B)\to \Cc^{\downarrow \Eq(p)}$ and  $K^{p'}\colon (\Cc\downarrow B')\to \Cc^{\downarrow \Eq(p')}$, morphisms in $(\Cc\downarrow B)$ and $(\Cc\downarrow B')$, and it is easily verified that these determine a morphism in $(\Ext(\Cc)\downarrow b)$ (by keeping in mind that $q'\colon (C',bf'q')\to (A',bf')$ is a regular epimorphism in $(\Cc\downarrow B)$, for arbitrary objects of  $\Ext(\Cc)^{\downarrow\Eq(p',p)}$ as in the above diagram).
\end{proof}

We still have to prove that if $\Gamma$ satisfies conditions (E4), (E5), (M) and (B), then $\Cov_{\Gamma}(\Cc)$ is a strongly $\E^1$-Birkhoff subcategory of $\Ext(\Cc)$. In order to obtain a left adjoint for  $H_1\colon \Cov_{\Gamma}(\Cc)\to \Ext(\Cc)$, we verify that the conditions of Theorem \ref{propositionreflectiveness} are satisfied:

\begin{proposition}\label{birkhoffimpliesadmissible}
If $\Gamma=(\Cc,\Xc,I,H,\eta,\E)$ satisfies conditions (E4), (E5), (M) and (B), then it is admissible.
\end{proposition}
\begin{proof}
As explained in the proof of Proposition 2.6 in \cite{EGV}, this is a consequence of the following two facts: for every morphism $f\colon A\to B$ in $\Ec$, the unit $\eta^B_f\colon (A,f)\to H^BI^B(A,f)$ is a regular epimorphism in the category $(\Cc\downarrow B)$ by conditions (B) and (M), and for every object $B\in\Cc$, the functor $H^B$ reflects isomorphisms because $H\colon \Xc\to \Cc$ is fully faithful and $\eta_B^*\colon (\Cc\downarrow HI(B))\to (\Cc\downarrow B)$ is monadic, once again by conditions (B) and (M). Indeed, for an arbitrary $(X,\varphi)\in (\Xc\downarrow I(B))$,  the triangular identity 
\[
H^B(\epsilon^B_{\varphi})\circ \eta^B_{H^B(X,\varphi)}=1_{H^B(X,\varphi)}
\] 
tells us that the (regular epic) adjunction unit  $\eta^B_{H^B(X,\varphi)}$ is a split monomorphism, hence an isomorphism. Consequently, $H^B(\epsilon_{\varphi}^B)$, and therefore also $\epsilon_{\varphi}^B$ is an isomorphism,  which proves that $H^B$ is fully faithful.
\end{proof}

Note that any commutative square (\ref{doubleext}) of morphisms in $\E$ may be considered as a commutative square in $(\Cc\downarrow B)$ in an obvious way. If the square is a pullback (in $(\Cc\downarrow B)$ or, equivalently, in $\Cc$) and all four of its morphisms are monadic extensions (hence, regular epimorphisms in $(\Cc\downarrow B)$) then it is necessarily also a pushout (in $(\Cc\downarrow B)$). Consequently, as soon as a square (\ref{doubleext}) of monadic extensions is in $\E^1$, it is a pushout square in $(\Cc\downarrow B)$. This will be useful to prove the following proposition:
\begin{proposition}\label{birkhoffimpliesregular}
If $\Gamma=(\Cc,\Xc,I,H,\eta,\E)$ satisfies conditions (E4), (E5), (M) and (B), then $I\colon \Cc\to \Xc$ preserves those pullbacks (\ref{pullback}) for which $f,g\in \Ec$ and $g$ is a split epimorphism.
\end{proposition}
\begin{proof}
Let us write $\Ext^2(\Cc)$ for the full subcategory of the ``double arrow'' category $\Arr^2(\Cc)=\Arr(\Arr(\Cc))$ determined by the class $\Ec^1$, and $\E^2$ for the class $(\E^1)^1$. By Lemma \ref{extensionsgoup}, first applied to the pair $(\Cc,\E)$ and next to the pair $(\Ext(\Cc),\Ec^1)$, we have that every split epimorphism in $\Ext^2(\Cc)$ lies in the class $\E^2$. 

Now, consider the diagram
\[
\xymatrix{
& HI(D) \ar@{.>}[dd] \ar@<.5 ex>[rr] && HI(A) \ar@<.5 ex>[ll]  \ar[dd]\\
D \ar[dd] \ar@<.5 ex>[rr] \ar[ru]^{\eta_D} && A \ar@<.5 ex>[ll] \ar[ru]_{\eta_A} \ar[dd]^<<<<<<<f &\\
& HI(C) \ar@<.5 ex>@{.>}[rr] && HI(B) \ar@{.>}@<.5 ex>[ll] \\
C \ar[ru]^{\eta_C} \ar@<.5 ex>[rr]^g && B \ar@<.5 ex>[ll] \ar[ru]_{\eta_B}&}
\]
in which the front square is a pullback for which $f,g\in \Ec$ and $g$ is a split epimorphism, and the back square is its image by the functor $HI\colon \Cc\to\Cc$. By the strong $\Ec$-Birkhoff property of $\Xc$, the left and right hand squares are in $\E^1$, hence we may consider the diagram as a split epimorphism in the category $\Ext^2(\Cc)$, whence a morphism in the class $\E^2$. It follows that the induced commutative square
\[
\xymatrix{
D \ar[r] \ar[d]_{\cong} & HI(D) \ar[d] \\
C\times_BA \ar[r] & HI(C)\times_{HI(B)}HI(A)}
\]
is in $\E^1$. In particular, it is a pushout in $(\Cc\downarrow HI(C)\times_{HI(B)}HI(A))$. Hence, the right hand vertical morphism is an isomorphism because the left hand one is so by assumption. 

This proves that the back square of the cube is a pullback. Since $H$ is a full inclusion, the result follows.
\end{proof}

Hence, we have that the conditions of Theorem \ref{propositionreflectiveness} hold as soon as  $\Gamma$ satisfies conditions (E4), (E5), (M) and (B), in which case $H_1\colon \Cov_{\Gamma}(\Cc)\to \Ext(\Cc)$ admits a left adjoint $I_1\colon \Ext(\Cc)\to \Cov_{\Gamma}(\Cc)$. We still have to prove that the square (\ref{birkhofsquare}) is in $\E^1$, for every $f\colon A\to B$ in $\E$. For this, we first show that every reflection unit is in $\E^1$. In fact, for this weaker assumptions on $\Gamma$ suffice:

\begin{lemma}\label{halfbirkhoff}
If  $\Gamma=(\Cc,\Xc,I,H,\eta,\E)$ satisfies the assumptions of Theorem \ref{propositionreflectiveness} as well as conditions (E4) and (E5), and if every reflection unit $\eta_A\colon A\to HI(A)$ is in $\E$, then every reflection unit $\eta^1_f\colon f\to I_1(f)$ is in $\E^1$.
\end{lemma}
\begin{proof}
Consider a morphism $f\colon A\to B$ in $\E$, and Diagram (\ref{bigdiagram}), where $\bar{f}=I_1(f)$. We must prove that the commutative square
\[
\xymatrix{
A \ar[r]^{\eta^1_f} \ar[d]_f & \bar{A} \ar[d]^{\bar{f}}\\
B \ar@{=}[r] & B
}
\]
is in $\Ec^1$, for which it clearly suffices to show that $\eta^1_f\colon A\to \bar{A}$ is in $\E$. Note that the square
\[
\xymatrix{
A\times_BA \ar[r]^-{\pi_1} \ar[d]_{\eta_{A\times_BA}} & A \ar[d]^{\eta_A}\\
HI(A\times_BA) \ar[r]_-{HI(\pi_1)} & HI(A)}
\]
is in $\E^1$ by condition (E5) and the assumption that the reflection units are in $\E$. This means, in particular, that in the subdiagram
\[
\xymatrix{
A\times_BA \ar[d]_{\pi_2} \ar[r]^-q & P_0 \ar[d]_p \ar[r]^{\bar{\pi}_1} & A \ar[d]^f  \\
A \ar[r]_{\eta^1_f} & \bar{A} \ar[r]_{\bar{f}} & B}
\]
of Diagram (\ref{bigdiagram}) we have that $q\in\E$. Since the right hand square is a pullback, $f\in\E$ implies $p\in\E$, and we can conclude from Conditions (E3) and (E4) that also $\eta^1_f\colon A\to \bar{A}$ lies in $\E$, as desired.
\end{proof}

\begin{lemma}\label{birkhoffgoesup}
If  $\Gamma=(\Cc,\Xc,I,H,\eta,\E)$ satisfies conditions (E4), (E5), (M) and (B), then $\Gamma_1=(\Ext(\Cc),\Cov_{\Gamma}(\Cc),I_1,H_1,\eta^1,\E^1)$ satisfies (B): $\Cov_{\Gamma}(\Cc)$ is a strongly $\E^1$-Birkhoff subcategory of $\Ext(\Cc)$.
\end{lemma}
\begin{proof}
Let us write $\Ext(\Xc)$ for the full subcategory of the arrow category $\Arr(\Xc)$ determined by the morphisms $\varphi\colon X\to Y$ in $\Xc$ for which $H(\varphi)\in\E$, and $\bar{H}$ for the obvious extension of $H\colon \Xc\to \Cc$ to $\Ext(\Xc)\to \Ext(\Cc)$. Since $HI(\E)\subseteq \E$, we also have that $I\colon \Cc\to\Xc$ extends to a functor $\bar{I}\colon \Ext(\Cc)\to \Ext(\Xc)$. Clearly, $\bar{H}$ is fully faithful just as $H$, and $\bar{I}$ is left adjoint to $\bar{H}$. We moreover have that $\bar{H}\bar{I}(\E^1)\subseteq\E^1$: indeed, consider  for any morphism  $a\to b$ in the class $\Ec^1$ the commutative square
\[
\xymatrix{
a \ar[r]^-{\bar{\eta}_a} \ar[d] & \bar{H}\bar{I}(a) \ar[d]\\
b \ar[r]_-{\bar{\eta}_b}& \bar{H}\bar{I}(b)}
\]
induced by the reflection unit; since from Lemma \ref{extensionsgoup} we know that the class $\E^1$ satisfies conditions (E3) and (E4), we need only note that $\bar{\eta}_b\in\E^1$ by the strong $\E$-Birkhoff property of $\Xc$.

Hence, we have a Galois structure $\bar{\Gamma}=(\Ext(\Cc),\Ext(\Xc),\bar{H},\bar{I},\bar{\eta},\E^1)$ which, moreover, satisfies the assumptions of Lemma \ref{halfbirkhoff}: indeed, from Lemmas \ref{extensionsgoup} and \ref{descentgoesup} we know that conditions (E4), (E5) and (M) are satisfied, and since pullbacks in $\Ext(\Cc)$ along morphisms in $\E^1$ are pointwise pulbacks in $\Cc$, $\bar{\Gamma}$ is easily verified to satisfy also the remaining assumptions. Furthermore, we clearly have that a morphism $(f',f)\colon a\to b$ in the class $\E^1$ is a normal extension (=covering) with respect to $\bar{\Gamma}$ if, and only if, both $f$ and $f'$ are normal extensions (=coverings) with respect to $\Gamma$, and the left adjoint $\bar{I}_1\colon \Ext(\Ext(\Cc))\to \Cov_{\bar{\Gamma}}(\Ext(\Cc))$ to the inclusion functor $\bar{H}_1\colon  \Cov_{\bar{\Gamma}}(\Ext(\Cc))\to  \Ext(\Ext(\Cc))$ is obtained simply by pointwise application of $I_1\colon \Ext(\Cc)\to \Cov_{\Gamma}(\Cc)$.

Applying Lemma  \ref{halfbirkhoff} to  $\bar{\Gamma}$ now yields the result.
\end{proof}

Finally, by combining Lemma's \ref{extensionsgoup}, \ref{descentgoesup} and \ref{birkhoffgoesup}, Propositions \ref{birkhoffimpliesadmissible} and  \ref{birkhoffimpliesregular}, and Theorem \ref{propositionreflectiveness}, we obtain our main result:

\begin{theorem}\label{maintheorem}
For every Galois structure  $\Gamma=(\Cc,\Xc,I,H,\eta,\E)$ satisfying conditions (E4), (E5), (M) and (B) (and (E1)--(E3) and (G)), there exists a Galois structure $\Gamma_1=(\Cc_1,\Xc_1,I_1,H_1,\eta^1,\Ec^1)$, once again satisfying these conditions, such that $\Cc_1$ is the full subcategory of $\Arr(\Cc)$ corresponding to $\Ec$, $\Xc_1$ is its full subcategory consisting exactly of the coverings with respect to $\Gamma$, and  $\Ec^1$ is as in Proposition \ref{propositiondext}.
 
Hence, by induction,  there exists for every $n\geq 1$ a Galois structure $\Gamma_n=(\Cc_n,\Xc_n,H_n,I_n,\eta^n,\Ec^n)$ such that $\Cc_n$ is the full subcategory of the ``$n$-fold arrow'' category $\Arr^n(\Cc)=\Arr(\Arr^{n-1}(\Cc))$ corresponding to $\Ec^{n-1}$, $\Xc_n$ is its full subcategory consisting exactly of the coverings with respect to $\Gamma_{n-1}$, and  $\E^n=(\E^{n-1})^1$.  (Here we assumed that $\Cc_0=\Cc$, $\Xc_0=\Xc$, $\E^0=\E$ and $\Arr^0(\Cc)=\Cc$.) 

Moreover, we have for each $\Gamma_n$ that every morphism in $\E^n$ is a monadic extension and that every covering is a normal extension. 
\end{theorem}


 \begin{remark}
Note that  a chain of Galois structures $\Gamma_n=(\Cc_n,\Xc_n,H_n,I_n,\eta^n,\Ec^n)$  may exist,  with  $\Cc_n=\Ext(\Cc_{n-1})$, $\E^n=(\E^{n-1})^1$ and where $\Xc_n$ consists of the coverings with respect to $\Gamma_{n-1}$,   even if the assumptions of Theorem \ref{maintheorem} are not satisfied.
To see this, take $\Cc$ arbitrary, put $\Xc=\Cc$ and choose $\Ec$ to be the class of all morphisms in $\Cc$. Then condition (M) is not satisfied, but  
  the existence of the Galois structures $\Gamma_n$ is readily observed. Note that a covering with respect to $\Gamma_n$ is simply an $n+1$-fold morphism.
 
Or, we may choose  $\Ec$ to be the class of regular epimorphisms in, say, an exact category $\Cc$, in which case condition (M) \emph{is} satisfied, and still take $\Xc=\Cc$. In this case, we do not only have the existence of the Galois structures $\Gamma_n$, but the last part of the theorem is satisfied as well: every morphism in $\E^n$ is a monadic extension and the classes of  coverings and of  normal extensions with respect to $\Gamma_n$  coincide (and, in fact, coincide with $\E^n$, as well as with the class of trivial coverings). However, condition (E5) fails for any  $\Cc$ which is not Mal'tsev.
%
%
%
%
%
%
%
%
%
%
 \end{remark}

%

\end{document}